\documentclass[11pt]{article}      
\usepackage{amsmath,amssymb,amscd,amsthm, graphicx, verbatim, setspace} 
\usepackage{authblk}

\theoremstyle{plain}
\newtheorem{thm}{Theorem}[section]
\newtheorem{lem}[thm]{Lemma}
\newtheorem{cor}[thm]{Corollary}

\newcommand{\f}{\operatorname{f}}

\newcommand{\E}{\operatorname{\mathbb{E}}}
\newcommand{\degi}{\operatorname{deg_>}}
\newcommand{\ex}{\operatorname{ex}}

\title{A generalization of the K\H{o}v\'{a}ri-S\'{o}s-Tur\'{a}n theorem}
\date{}
\author{Jesse Geneson\\
\small\tt geneson@gmail.com
}
\begin{document}
\maketitle

\begin{abstract} 
We present a new proof of the K\H{o}v\'{a}ri-S\'{o}s-Tur\'{a}n theorem that $\ex(n, K_{s,t}) = O(n^{2-1/t})$ for $s, t \geq 2$. The new proof is elementary, avoiding the use of convexity. For any $d$-uniform hypergraph $H$, let $\ex_d(n,H)$ be the maximum possible number of edges in an $H$-free $d$-uniform hypergraph on $n$ vertices. Let $K_{H, t}$ be the $(d+1)$-uniform hypergraph obtained from $H$ by adding $t$ new vertices $v_1, \dots, v_t$ and replacing every edge $e$ in $E(H)$ with $t$ edges $e \cup \left\{v_1\right\},\dots, e \cup \left\{v_t\right\}$ in $E(K_{H, t})$. If $H$ is the $1$-uniform hypergraph on $s$ vertices with $s$ edges, then $K_{H, t} = K_{s, t}$. \

We prove that $\ex_{d+1}(n,K_{H,t}) = O(\ex_d(n, H)^{1/t} n^{d+1-d/t} + t n^d)$ for any $d$-uniform hypergraph $H$ with at least two edges such that $\ex_d(n, H) = o(n^d)$. Thus $\ex_{d+1}(n,K_{H,t}) = O(n^{d+1-1/t})$ for any $d$-uniform hypergraph $H$ with at least two edges such that $\ex_d(n, H) = O(n^{d-1})$, which implies the K\H{o}v\'{a}ri-S\'{o}s-Tur\'{a}n theorem in the $d = 1$ case. This also implies that $\ex_{d+1}(n, K_{H,t}) = O(n^{d+1-1/t})$ when $H$ is a $d$-uniform hypergraph with at least two edges in which all edges are pairwise disjoint, which generalizes an upper bound proved by Mubayi and Verstra\"{e}te (JCTA, 2004). We also obtain analogous bounds for 0-1 matrix Tur\'{a}n problems.
\end{abstract}

\section{Introduction}
The K\H{o}v\'{a}ri-S\'{o}s-Tur\'{a}n theorem is one of the most famous results in extremal combinatorics \cite{kst, fur1, fur}. The theorem states that the maximum number of edges in a $K_{s, t}$-free graph of order $n$ is $O(n^{2-1/t})$. There are multiple known proofs of this theorem, including a standard double-counting proof that uses Jensen's inequality, as well as a proof that uses dependent random choice and Jensen's inequality \cite{aks}. For a student to fully understand past proofs of the K\H{o}v\'{a}ri-S\'{o}s-Tur\'{a}n theorem, they would need to understand convexity, which would require calculus background.

In this paper, we prove the K\H{o}v\'{a}ri-S\'{o}s-Tur\'{a}n theorem without using calculus or Jensen's inequality. Instead we use a method based on Nivasch's bounds on Davenport-Schinzel sequences \cite{niv} and Alon et al.'s bounds on interval chains \cite{aknss}. This new proof gives a simple way to teach the proof of the K\H{o}v\'{a}ri-S\'{o}s-Tur\'{a}n theorem to students with no calculus background, and the same method can be used to prove a generalization of the K\H{o}v\'{a}ri-S\'{o}s-Tur\'{a}n theorem for uniform hypergraphs.

In \cite{niv}, Nivasch found upper bounds on the maximum possible lengths of Davenport-Schinzel sequences using two different methods. Both methods gave the same bounds, but the first method was more like the proofs in past papers on Davenport-Schinzel sequences, and the second method was similar to proofs about interval chains in \cite{aknss}. The second method in \cite{niv} was much simpler than the first for proving bounds on Davenport-Schinzel sequences, so we imitate the second method here for graph and hypergraph Tur\'{a}n problems.

Let $\ex_d(n, H)$ denote the maximum number of edges in an $H$-free $d$-uniform hypergraph on $n$ vertices. Let $K_{H, t}$ be the $(d+1)$-uniform hypergraph obtained from $H$ by adding $t$ new vertices $v_1, \dots, v_t$ and replacing every edge $e$ in $E(H)$ with $e \cup \left\{v_1\right\},\dots, e \cup \left\{v_t\right\}$ in $E(K_{H, t})$. For example, if $H$ is the $1$-uniform hypergraph of order $s$ with $s$ edges, then $K_{H, t} = K_{s, t}$. Mubayi and Verstra\"{e}te \cite{MV} proved that $\ex_3(n, K_{H,t}) = O(n^{3-1/t})$ when $H$ is a $2$-uniform hypergraph in which all edges are pairwise disjoint.

In Section \ref{s:ds}, we provide an elementary proof that $\ex_{d+1}(n,K_{H,t}) = O(\ex_d(n, H)^{1/t} n^{d+1-d/t} + t n^d)$ for any $d$-uniform hypergraph $H$ with at least two edges such that $\ex_d(n, H) = o(n^{d})$, giving an alternative proof of the K\H{o}v\'{a}ri-S\'{o}s-Tur\'{a}n theorem when $H$ is the $1$-uniform hypergraph of order $s$ with $s$ edges. As a corollary, this implies that $\ex_{d+1}(n, K_{H,t}) = O(n^{d+1-1/t})$ when $H$ is a $d$-uniform hypergraph with at least two edges in which all edges are pairwise disjoint, generalizing the upper bound of Mubayi and Verstra\"{e}te. In Section \ref{s:01}, we discuss analogous results about $d$-dimensional 0-1 matrices that can be proved with similar methods.

\section{The letter method}\label{s:ds}

An \emph{ordered} $d$-uniform hypergraph is a $d$-uniform hypergraph with a linear order on the vertices. We define a \emph{lettered} $d$-uniform hypergraph as the structure obtained from labeling each edge of an ordered $d$-uniform hypergraph with a letter such that two edges can be labeled with the same letter only if they have the same greatest vertex. Given a $d$-uniform hypergraph $H$, we say that a lettered $d$-uniform hypergraph is $H$-free if its underlying $d$-uniform hypergraph is $H$-free.

For any $d$-uniform hypergraph $H$, let $\f_{d}(n, k, H)$ denote the maximum possible number of distinct letters in an $H$-free lettered $d$-uniform hypergraph on $n$ vertices in which every letter occurs at least $k$ times.

The next lemma is analogous to inequalities in \cite{niv, ck, gseq, ff} and is proved similarly.

\begin{lem}\label{klem}
For all positive integers $n, k$ and $d$-uniform hypergraphs $H$, we have $\ex_{d}(n, H) \leq k(\f_{d}(n, k, H)+n)$.
\end{lem}

\begin{proof}
Start with a $d$-uniform $H$-free hypergraph $Q$ with $\ex_{d}(n, H)$ edges. Order the vertices of $Q$ arbitrarily. For each vertex $v$ in $V(Q)$ in order from greatest to least, label the unlabeled edges adjacent to $v$ in any order with letters $v_0, v_1, \dots$, only using each letter $v_i$ exactly $k$ times and deleting up to $k-1$ remaining edges adjacent to $v$ if $k$ does not divide the total number of edges in which $v$ is the greatest vertex. Observe that the new lettered hypergraph has at most $\f_{d}(n, k, H)$ distinct letters with every letter occurring exactly $k$ times, and it is $H$-free.
\end{proof}

When combined with Lemma \ref{klem}, the next lemma will complete our proof of the generalization of the K\H{o}v\'{a}ri-S\'{o}s-Tur\'{a}n theorem. We use Stirling's bound in the proof of the next lemma, but it is not actually necessary. We explain after the proof how the use of Stirling's bound can be replaced with an elementary one-sentence argument. 

\begin{lem}\label{proofg}
For $t \geq 2$, $H$ a $d$-uniform hypergraph with at least two edges such that $\ex_d(n, H) = o(n^{d})$, and $k = \lceil 2 e n^{d(1-1/t)} \ex_d(n, H)^{1/t} \rceil$, we have $\f_{d+1}(n, k, K_{H,t}) =O(t (\frac{n^d}{\ex_d(n, H)})^{1/t})$.
\end{lem}

\begin{proof}
Suppose for contradiction that there exists a $K_{H,t}$-free lettered $(d+1)$-uniform hypergraph $Q$ on $n$ vertices with $r = \lfloor \frac{t}{e}(\frac{n^d}{\ex_d(n, H)})^{1/t} \rfloor$ distinct letters in which every letter occurs at least $k$ times. Suppose that $n$ is sufficiently large so that $r \geq \frac{t}{2e}(\frac{n^d}{\ex_d(n, H)})^{1/t}$. Delete edges of $Q$ until every letter occurs exactly $k$ times. 

For each $d$-subset $z$ of $V(Q)$, define $\degi(z)$ to be the number of edges in $E(Q)$ that contain all of the vertices in $z$ and a greater vertex in the ordering. Let $p$ be the number of $d$-subsets $z$ of $V(Q)$ with $\degi(z) > 0$. The number of $t$-tuples of edges in $E(Q)$ that have the same $d$ least vertices is equal to $\sum_{z: \degi(z) \geq t} \binom{\degi(z)}{t}$, which is at most $\binom{r}{t} \ex_d(n, H)$, or else $Q$ would contain a copy of $K_{H,t}$. This follows by the pigeonhole principle, since every $t$-tuple of edges in $E(Q)$ that have the same $d$ least vertices must have different letters on each edge.

Then $k r = \sum_{z} \degi(z)$ and $(t-1)p \geq k r - \sum_{z: \degi(z) \geq t} (\degi(z)-t+1) \geq k r - \sum_{z: \degi(z) \geq t} \binom{\degi(z)}{t} \geq k r - \binom{r}{t} \ex_d(n, H) \geq k r - \frac{r^t \ex_d(n, H)}{t!} \geq t n^d - \frac{(\frac{t}{e})^t}{t!} n^d > (t-1)n^d$, where the last inequality follows from Stirling's bound. However $p \leq \binom{n}{d}$, a contradiction.
\end{proof}

\begin{thm}\label{thmseq}
For fixed $t \geq 2$ and $d$-uniform hypergraph $H$ with at least two edges such that $\ex_d(n, H) = o(n^{d})$, we have $\ex_{d+1}(n, K_{H,t}) = O(\ex_d(n, H)^{1/t} n^{d+1-d/t} + t n^d)$.
\end{thm}

The use of Stirling's bound in Lemma \ref{proofg} may seem to make the proof non-elementary, but it was unnecessary. All we need is that there exists some constant $c$ such that $t! > \left(\frac{t}{c}\right)^{t}$ for all $t > 1$, and then we can replace each $e$ in the last proof with $c$. Proving this for $c = 8$: it is clearly true for $t \leq 8$, and if we assume that $t! > \left(\frac{t}{8}\right)^{t}$, then we also have $(2t)! > \left(\frac{t}{4}\right)^t 4^t \left(\frac{t}{8}\right)^{t} = \left(\frac{t}{4}\right)^{2t} 2^t > \left(\frac{t}{4}\right)^{2t}$, and $(2t+1)! = (2t+1)(2t)! > \left(\frac{t}{4}\right)^{2t} 2^t (2t+1) > \left(\frac{t}{4}\right)^{2t} \left(1+\frac{1}{2t}\right)^{2t} \left(\frac{2t+1}{8} \right)= \left(\frac{2t+1}{8}\right)^{2t+1}$. Thus the whole proof is elementary.

\begin{cor}
If $H$ is a $d$-uniform hypergraph with at least two edges in which all edges are pairwise disjoint, then $\ex_{d+1}(n, K_{H,t}) = O(n^{d+1-1/t})$.
\end{cor}

\begin{proof}
If $H$ is a $d$-uniform hypergraph in which all edges are pairwise disjoint, then $\ex_d(n, H) = O(n^{d-1})$, so this bound follows from Theorem \ref{thmseq}.
\end{proof}

The last corollary yields the bound of Mubayi and Verstra\"{e}te from \cite{MV} when $d = 2$.

\section{0-1 matrices}\label{s:01}

Using the same method, we can get similar bounds for Tur\'{a}n-type problems on $d$-dimensional 0-1 matrices. In order to state these results, we introduce more terminology. We say that $d$-dimensional 0-1 matrix $A$ \emph{contains} $d$-dimensional 0-1 matrix $B$ if some submatrix of $A$ can be turned into $B$ by changing some number of ones to zeroes. Otherwise $A$ \emph{avoids} $B$. For any $d$-dimensional 0-1 matrix $Q$, define $\ex(n, Q, d)$ to be the maximum number of ones in a $d$-dimensional 0-1 matrix of sidelength $n$ that avoids $Q$. 

As with the case of $d$-uniform hypergraphs, most of the past research on the topic of $d$-dimensional 0-1 matrices has focused on when $d = 2$. We mention several results for $d = 2$ that have been generalized to higher values of $d$. For example, Klazar and Marcus \cite{KM} proved that $\ex(n, P, d) = O(n^{d-1})$ for every $d$-dimensional permutation matrix $P$, generalizing the result of Marcus and Tardos \cite{MT}. Geneson and Tian \cite{GTmat} sharpened this bound by proving that $\ex(n, P, d) = 2^{O(k)}n^{d-1}$ for $d$-dimensional permutation matrices $P$ of sidelength $k$, generalizing a result of Fox \cite{fox}. Geneson and Tian also proved that $\ex(n, P, d) = O(n^{d-1})$ for every $d$-dimensional double permutation matrix $P$, generalizing the upper bound in \cite{gduluth}. 

In order to state the next result, we define $Q_{P, t}$ to be the $(d+1)$-dimensional 0-1 matrix obtained from the $d$-dimensional 0-1 matrix $P$ by stacking $t$ copies of $P$ with the same orientation in the direction of the new dimension. For example if $P$ is the $4 \times 1$ matrix of all ones, then $Q_{P, t}$ is the $4 \times t$ matrix of all ones. 

\begin{thm}\label{01pancakes}
\begin{enumerate}
\item For fixed $t$ and $d$-dimensional 0-1 matrix $P$ with at least two ones, $\ex(n, Q_{P,t},d+1) = O(\ex(n, P, d)^{1/t} n^{d+1-d/t})$.

\item For any $d$-dimensional 0-1 matrix $P$ with $\ex(n, P, d) = O(n^{d-1})$, we have $\ex(n, Q_{P,t},d+1) = O(n^{d+1-1/t})$. In particular, $\ex(n, Q_{P,2}, d+1) = \Theta(n^{d+1/2})$ for any $d$-dimensional 0-1 matrix $P$ with at least two ones such that $\ex(n, P, d) = O(n^{d-1})$. Moreover, $\ex(n, Q_{P,3}, d+1) = \Theta(n^{d+2/3})$ for any $d$-dimensional 0-1 matrix $P$ with at least three ones differing in the first coordinate such that $\ex(n, P, d) = O(n^{d-1})$.
\end{enumerate}
\end{thm}

\begin{proof}
The upper bounds follow from using the letter method as in the last section. The lower bounds follow from stacking copies of the known lower bound constructions for extremal functions of forbidden $2 \times 2$ and $3 \times 3$ all-ones matrices \cite{brown, polarity, fur2, fur}.
\end{proof}

Permutation matrices and double permutation matrices $P$ with at least three ones are some examples for which Theorem \ref{01pancakes} gives sharp bounds on $\ex(n, Q_{P,2},d+1)$ and $\ex(n, Q_{P,3},d+1)$ up to a constant factor \cite{MT, gduluth, KM, GTmat}.

\section{Concluding remarks}

The standard double-counting method used to prove the K\H{o}v\'{a}ri-S\'{o}s-Tur\'{a}n theorem can also be used to prove the bounds in Theorem \ref{thmseq} and \ref{01pancakes}. We did not include this method since it uses convexity, and it gives the same bounds up to a constant factor as the letter method. Dependent random choice can also be used to obtain the same bounds for uniform hypergraphs up to a constant factor when $\ex_d(n, H) = O(n^{d-1})$, and it can be applied to a larger family of hypergraphs that contains the family of $K_{H,t}$, but it gives a worse bound than the letter method when $\ex_d(n, H) = \omega(n^{d-1})$. The next lemma is a generalization of the dependent random choice lemma from \cite{aks} and \cite{fs}. In the next lemma, we call a vertex $v$ and a $d$-subset $T$ of vertices of a $(d+1)$-uniform hypergraph $G$ \emph{neighbors} if there is some edge of $G$ that contains $v$ and all of the vertices of $T$. For each vertex $v$ and set of vertices $S$, we define $N(v)$ to be the set of $d$-subsets of vertices that are neighbors with $v$, and we define $N(S)$ to be the set of $d$-subsets of vertices that are neighbors with every vertex in $S$.

\begin{lem}\label{drc}
Let $G = (V, E)$ be a $(d+1)$-uniform hypergraph with $|V| = n$ vertices and $|E| = m$ edges. If there is a positive integer $t$ such that $n \binom{n}{d}^{-t} \left( \frac{m}{n} \right)^t - \binom{n}{r} \left(\frac{x}{\binom{n}{d}}\right)^t \geq a$, then $G$ contains a subset $A$ of at least $a$ vertices such that every $r$ vertices in $A$ have at least $x$ common neighbors among the $d$-subsets of $V$.
\end{lem}

\begin{proof}
Pick a set $T$ of $d$-subsets of vertices of $V$, choosing $t$ $d$-subsets uniformly at random with repetition. Let $B = N(T)$, and let $X$ be the cardinality of $B$. Then $\E[X] = \sum_{v \in V} \left( \frac{|N(v)|}{\binom{n}{d}} \right)^t = \binom{n}{d}^{-t} \sum_{v \in V}|N(v)|^t \geq n \binom{n}{d}^{-t} \left( \frac{\sum_{v \in V} |N(v)|}{n} \right)^t \geq n \binom{n}{d}^{-t} \left( \frac{m}{n} \right)^t$, where the second-to-last inequality used Jensen's inequality. 

Let $Y$ be the random variable for the number of subsets $S \subset B$ of size $r$ with fewer than $x$ common neighbors among the $d$-subsets of vertices of $V$. The probability that an arbitrary $r$-subset $S$ is a subset of $B$ is $\left(\frac{|N(S)|}{\binom{n}{d}}\right)^t$, so $\E[Y] \leq \binom{n}{r} \left(\frac{x}{\binom{n}{d}}\right)^t$. 

Thus by linearity of expectation, $\E[X-Y] \geq n \binom{n}{d}^{-t} \left( \frac{m}{n} \right)^t - \binom{n}{r} \left(\frac{x}{\binom{n}{d}}\right)^t \geq a$. Thus there exists a choice of $T$ for which the corresponding set $B$ of cardinality $X$ satisfies $X-Y \geq a$, so we can remove $Y$ vertices from $B$ to produce a new subset $A$ so that all $r$-subsets of $A$ have at least $x$ common neighbors among the $d$-subsets of $V$.
\end{proof}

We can use Lemma \ref{drc} to get upper bounds for a more general family of $(d+1)$-uniform hypergraphs that contains the family of $K_{H,t}$. The next theorem describes one such family.

\begin{thm}
For any $d$-uniform hypergraph $H$, let $K_{H, t, s, r}$ be the $(d+1)$-uniform hypergraph obtained by starting with $s$ vertices $S = \left\{v_1, \dots, v_s\right\}$, making $r$ disjoint copies $H_{T_1}, \dots, H_{T_r}$ of $H$ for each $t$-subset $T$ of vertices of $S$, and replacing each edge $e$ in each $H_{T_i}$ with $t$ edges of the form $e \cup \left\{ u \right\}$ for each $u \in T$. Then $\ex_{d+1}(n, K_{H,t,s,r}) = O(\ex_d(n, H) n^{2-1/t})$ for any $d$-uniform hypergraph $H$ with at least two edges and any integers $s \geq t \geq 2$ and $r > 0$. For any $d$-uniform hypergraph $H$ with at least two edges such that $\ex_d(n, H) = O(n^{d-1})$ and any integers $s \geq t \geq 2$ and $r > 0$, we have $\ex_{d+1}(n, K_{H,t,s,r}) = O(n^{d+1-1/t})$.
\end{thm}

Note that $K_{H, t, t, 1} = K_{H, t}$, and that the letter method also works to show that $\ex_{d+1}(n, K_{H,t,t,r}) = O(n^{d+1-1/t})$ for any integers $t \geq 2$, $r > 0$, and $d$-uniform hypergraph $H$ with at least two edges such that $\ex_d(n, H) = O(n^{d-1})$. It would be interesting to see if the letter method is useful for other Tur\'{a}n-type problems, and what else can be said about $\ex_{d+1}(n, K_{H,t})$ in general.


\begin{thebibliography}{}
\bibitem{aknss}  N. Alon, H. Kaplan, G. Nivasch, M. Sharir, and S. Smorodinsky, Weak $\epsilon$-nets and interval chains, J. ACM, 55, article 28, 32 pages, 2008.
\bibitem{aks} N. Alon, M. Krivelevich, B. Sudakov, Tur\'{a}n numbers of bipartite graphs and related Ramsey-type questions, Combin. Probab. Comput. 12 (2003), 477-494.
\bibitem{brown} W. Brown, On graphs that do not contain a Thomsen graph, Canad. Math. Bull. 9 (1966) 281-285.
\bibitem{ck} J. Cibulka and J. Kyn\v{c}l. Tight bounds on the maximum size of a set of permutations with bounded vc-dimension. Journal of Combinatorial Theory Series A, 119: 1461-1478, 2012.
\bibitem{polarity} P. Erd\H{o}s, A. R\'{e}nyi, and Vera T. S\'{o}s: On a problem of graph theory, Stud Sci. Math. Hung. 1 (1966) 215-235.
\bibitem{fox} J. Fox, Stanley-Wilf limits are typically exponential,  arXiv:1310.8378, 2013.
\bibitem{fs} J. Fox, B. Sudakov. Dependent random choice. Random Structures Algorithms 38 (2011) 68-99.
\bibitem{fur1} Z. F\H{u}redi, An upper bound on Zarankiewicz problem, Combin. Probab. Comput. 5 (1996) 29-33.
\bibitem{fur2} Z. F\H{u}redi, New asymptotics for bipartite Tur\'{a}n numbers, J. Combin. Theory Ser. A 75 (1996) 141-144.
\bibitem{fur} Z. F\H{u}redi and M. Simonovits. The history of degenerate (bipartite) extremal graph problems. In Erd\H{o}s centennial, Bolyai Soc. Math. Stud. 25 (2013) 169-264.
\bibitem{gseq} J. Geneson, A Relationship Between Generalized Davenport-Schinzel Sequences and Interval Chains. Electr. J. Comb. 22(3): P3.19, 2015.
\bibitem{gduluth} J. Geneson, Extremal functions of forbidden double permutation matrices, J. Combin. Theory Ser. A, 116 (7) (2009), 1235-1244.
\bibitem{ff} J. Geneson, Forbidden formations in multidimensional 0-1 matrices. European J. of Comb. 78, 147-154, 2019.
\bibitem{GTmat} J. Geneson, P. Tian, Extremal functions of forbidden multidimensional matrices. Discrete Mathematics 340(12): 2769-2781, 2017.
\bibitem{KM} M. Klazar and A. Marcus, Extensions of the linear bound in the Furedi-Hajnal conjecture, Advances in Applied Mathematics, 38 (2) (2007), 258-266.
\bibitem{kst} T. K\H{o}vari, V. T. S\'{o}s, and P. Tur\'{a}n. On a problem of K. Zarankiewicz. Colloquium Math., 3: 50-57, 1954.
\bibitem{MT} A. Marcus and G. Tardos, Excluded permutation matrices and the Stanley-Wilf conjecture, J. Combin. Theory Ser. A, 107 (1) (2004), 153-160.
\bibitem{MV} D. Mubayi, J. A. Verstra\"{e}te, A hypergraph extension of the bipartite Tur\'{a}n problem, J. Combin. Theory Ser. A 106 (2004) 237-253.
\bibitem{niv}  G. Nivasch. Improved bounds and new techniques for Davenport-Schinzel sequences and their generalizations. J. ACM, 57(3), 2010.

\end{thebibliography}
\end{document}